\documentclass[letterpaper, 10 pt, conference]{ieeeconf}  
\IEEEoverridecommandlockouts{}                            
\usepackage{myfavmacros}
\let\labelindent\relax
\usepackage[inline]{enumitem}
\usepackage[justification=centering]{caption}
\usepackage[vlined,ruled,linesnumbered,norelsize]{algorithm2e}
\usepackage{tabularx,booktabs}
\usepackage[caption=false, font=footnotesize]{subfig}
\usepackage{empheq}
\usepackage{lipsum}
\usepackage[%
backref=true,%
natbib=true,%
maxbibnames=99,%
doi=false,%
url=false,%
giveninits=true,%
style=numeric,
uniquelist=false]{biblatex}
\usepackage{hyperref}
\title{\LARGE \bf
Wasserstein Two-Sided Chance Constraints with An Application to Optimal Power Flow
}

\author{Haoming Shen$^1$ and Ruiwei Jiang$^1$
\thanks{$^1$University of Michigan, Department of Industrial and Operations Engineering. E-mail: {\tt\small \{hmshen,ruiwei\}@umich.edu}}%
}

\begin{document}
\maketitle
\thispagestyle{empty}
\pagestyle{empty}

\begin{abstract}

As a natural approach to modeling system safety conditions, chance constraint (CC) seeks to satisfy a set of uncertain inequalities individually or jointly with high probability. Although a joint CC offers stronger reliability certificate, it is oftentimes much more challenging to compute than individual CCs. Motivated by the application of optimal power flow, we study a special joint CC, named two-sided CC. We model the uncertain parameters through a Wasserstein ball centered at a Gaussian distribution and derive a hierarchy of conservative approximations based on second-order conic constraints, which can be efficiently computed by off-the-shelf commercial solvers. In addition, we show the asymptotic consistency of these approximations and derive their approximation guarantee when only a finite hierarchy is adopted. We demonstrate the out-of-sample performance and scalability of the proposed model and approximations in a case study based on the IEEE 118-bus and 3120-bus systems.

\end{abstract}

\section{Introduction}

Chance constraint (\CC{}) is a natural approach for modeling safety conditions of a system under uncertainty. \CC{} models the safety conditions as a set of inequalities and the underlying uncertainty as a random vector. Then, it requires to satisfy these inequalities individually or jointly with high probability. For linear inequalities, \CC{} takes the form
\begin{align}
\Probt \left[ A(x) \tilde{\xi} \leq b(x) \right]
& \geq 1 - \epsilon, \tag{\textbf{CC}}
\end{align}
where \(x \in \reals^n\) are decision or design variables, \(\tilde{\xi}\) is a random vector supported on \(\Xi := \reals^m\),
\(A(\cdot) \colon \reals^n \to \reals^{q \times m}\) and
\(b(\cdot) \colon \reals^n \to \reals^q\) are affine mappings, \(\Probt\)
is the probability distribution of \(\tilde{\xi}\), and
\((1 - \epsilon) \in (1/2, 1)\) is a risk threshold that is usually close to
one, \eg{}, \(0.95\). We call~(\textbf{CC})  \textit{individual}
if \(q = 1\) and \textit{joint} if \(q \geq 2\).

With its study dating back to the  \(1950\)s~\cite{charnes-1959-chanc-const-progr,
  charnes-1958-cost-horiz, miller-1965-chanc-const, prekopa-1970-on-proba},~\CC{} finds a wide range of applications in,~\eg{}, power
system~\cite{wang-2011-chanc-const}, vehicle
routing~\cite{stewart-1983-stoch-vehic-routin}, portfolio
management~\cite{li-1995-insur-inves}, scheduling~\cite{deng-2016-decom-algor},
and facility location~\cite{miranda-2006-simul-inven}. Despite its popularity in
real-world applications, (\textbf{CC}) is in general challenging to compute because of its non-convexity and the NP-hardness of evaluating probability through multi-dimensional integral~\cite{nemirovski-2007-convex-approx, khachiyan-1989-probl-calcul}. In particular, a joint (\textbf{CC}) is oftentimes much more challenging to compute than individual (\textbf{CC})s. For example, the individual (\textbf{CC}) admits a convex or conic reformulation in various settings (see, e.g.,~\cite{luedtke-2008-integ-progr,ghaoui-2003-worst-case,zymler-2011-distr-robus,zhang-2018-ambig-chanc,li-2019-ambig-risk}), while the corresponding results for the joint (\textbf{CC})s are unavailable to date (see, e.g.,~\cite{xie-2016-deter-refor,hanasusanto-2017-ambig-joint,xie-2019-distr-robus}). Consequently, convex and tractable approximations of joint (\textbf{CC})s with performance guarantee are crucial for its practical applications.


In this paper, we consider a special joint (\textbf{CC}), which we call a two-sided chance constraint (\textbf{2SCC}) of the form
\begin{align*}
\Probt{} \left[ \ell \leq x^{\top} \tilde{\xi} \leq u \right] \geq 1 - \epsilon,
\end{align*}
where \((x, \ell, u)\) are decision variables. (\textbf{2SCC}) is a joint (\textbf{CC}) because it requires both inequalities to hold jointly, but meanwhile it is special because the two inequalities share the term \(x^{\top}\tilde{\xi}\) (up to the contrary sign). The particular form of (\textbf{2SCC}) arises from optimal power flow (OPF) when modeling the lower/upper limits of power generation and those of power flow in a transmission line (see formulation~\eqref{eq:exps-ts-drc} in Section~\ref{sec:exps}), as well as from other applications including hydrothermal unit commitment~\cite{van2018exact} and robust regression~\cite{fathabad2021outlier}.

(\textbf{2SCC}) was first proposed by~\cite{lubin-2015-two-sided}, where $\Probt$ is assumed to be Gaussian. In this case,~\cite{lubin-2015-two-sided} showed that (\textbf{2SCC}) produces a convex feasible region and derived outer conic approximations with approximation guarantee. Later,~\cite{fathabad-2021-tight-conic} considered a more general case, in which \(\Probt\) is a mixture of \(K\) Gaussian
distributions sharing the same covariance matrix and \(\epsilon\) is
sufficiently close to \(0\). Then,~\cite{fathabad-2021-tight-conic} derived an asymptotically tight conic approximation for (\textbf{2SCC}) using a piecewise linear approximation of the standard Gaussian cumulative distribution function (CDF). In most real-world applications, however, the (true) distribution $\Probt$ is not available. Under such circumstance, a common choice is to replace $\Probt$ with a crude estimate $\Prob$, which can be an empirical distribution constructed from past observations of the uncertain parameters~\cite{luedtke-2008-sampl-approx}, or a Gaussian distribution, whose mean and covariance matrix can in turn be estimated empirically~\cite{lubin-2015-two-sided}. Unfortunately, such a $\Prob$ is likely to  misrepresent $\Probt$ and the decisions thus produced have disappointing out-of-sample performance. This motivates us to consider alternative estimates of $\Probt$, or more formally, a Wasserstein ball
\begin{align*}
  \mathcal{P} := \Set{ \QProb \in \mathcal{Q}_0 \colon d_W(\QProb, \Prob) \leq \delta }
\end{align*}
around $\Prob$, which consists of all distributions that are close enough to $\Prob$. Above, \(\mathcal{Q}_0\) is the set of all distributions supported on
\(\Xi\), \(\delta > 0\) is a pre-specified radius of the Wasserstein ball, and
\(d_W(\cdot, \cdot) \colon \mathcal{Q}_0 \times \mathcal{Q}_0 \to \reals_+\)
denotes the Wasserstein distance between two distributions, such that for any \(\Prob_1\) and
\(\Prob_2\) in \(\mathcal{Q}_0\),
\begin{align*}
  d_W(\Prob_1, \Prob_2) := \inf_{\QProb_0 \sim (\Prob_1, \Prob_2)}
  \Expect_{\QProb_0} \left[ \norm{\tilde{X}_1 - \tilde{X}_2} \right],
\end{align*}
where \(\tilde{X}_1\) and \(\tilde{X}_2\) are two random variables following \(\Prob_1\) and \(\Prob_2\), respectively, \(\QProb_0\) is a
coupling of \(\Prob_1\) and \(\Prob_2\), and \(\norm{\cdot}\) is a norm. Accordingly, we robustify (\textbf{2SCC}) by satisfying the chance constraint with regard to all distributions in $\mathcal{P}$, yielding the following two-sided distributionally robust chance constraint (\textbf{2DRC}),
\begin{align*}
  \mathcal{Z} := \Set{ (x, \ell, u) \in \reals^{n+2} \colon
  \inf_{\QProb \in \mathcal{P}} \QProb \left[ \ell \leq x^{\top} \tilde{\xi} \leq u \right] \geq 1 - \epsilon
  }.
\end{align*}
The recent literature has witnessed an increasing interest in the convexity and tractable reformulations of distributionally robust chance
constraints (\textbf{DRC}). For example,~\cite{ghaoui-2003-worst-case,calafiore-2006-distr-robus,zhang-2018-ambig-chanc} derived second-order conic representations for \emph{individual} (\textbf{DRC}) when the uncertainty is modeled by its mean and covariance matrix. Using the same model of uncertainty,~\cite{xie-2017-distr-robus} derived a second-order conic representation for \emph{two-sided} (\textbf{DRC}). Similar results were obtained when shape information (e.g., unimodality and log-concavity) is incorporated~\cite{hanasusanto2015decision,li-2019-ambig-risk,li2016distributionally}. Different from these works, we model the uncertainty using the Wasserstein ball $\mathcal{P}$, which is less conservative than the moment approaches~\cite{mohajerin-2018-data-driven,kuccukyavuz2021chance}.

Convexity and tractable reformulations for \emph{joint} (\textbf{DRC}) are much scarcer to date. For example,~\cite{hanasusanto-2017-ambig-joint,xie-2016-deter-refor,xie-2019-optim-bonfer} showed NP-hardness results of these constraints unless the model of uncertainty falls into certain special settings. In addition,~\cite{xie-2019-distr-robus,ho-nguyen-2020-stron-formul,ho-nguyen-2021-distr-robus,shen-2020-chanc-const} recast joint (\textbf{DRC}) as (deterministic) mixed-integer programs. On the contrary,~\cite{shen-2021-convex-chanc} showed that a joint (\textbf{DRC}) produces a convex feasible region if (i) the estimate $\Prob$ in $\mathcal{P}$ is chosen to be log-concave and (ii) the uncertainty is decoupled from decision variables in the chance constraint, \ie{}, $A(x) \equiv A$ is independent of $x$ in the definition of (\textbf{CC}). In the case of (\textbf{2DRC}), this implies that $x$ becomes constants, while $(\ell, u)$ remain decision variables, in  $\mathcal{Z}$. In this paper, we extend~\cite{shen-2021-convex-chanc} by allowing the coupling of $x$ and $\tilde{\xi}$ in (\textbf{2DRC}). Our main contributions include
\begin{enumerate}
\item We show that \(\mathcal{Z}\) is convex. In addition, we derive a hierarchy of conservative approximations for \(\mathcal{Z}\) based on second-order conic constraints, which facilitates efficient computation through off-the-shelf commercial solvers.
\item We show that these approximations are asymptotically tight and derive their non-asymptotic approximation guarantee, when only a finite hierarchy is adopted.
\item Using the~\OPF{} problem and an IEEE 118-bus system, we numerically demonstrate the out-of-sample performance of the proposed (\textbf{2DRC}) over the alternative (\textbf{CC}), which does not model robustness. In addition, we demonstrate the scalability of the approximations using IEEE systems with up to 3120 buses.
\end{enumerate}
The rest of the paper is organized as follows: In Section~\ref{sec:cvx}, we derive a convex representation for \(\mathcal{Z}\).
In Section~\ref{sec:conic-approx}, we construct an asymptotically exact conic
approximation of \(\mathcal{Z}\) and derive non-asymptotic approximation
guarantee. In Section~\ref{sec:exps}, we numerically demonstrate the effectiveness and scalability of our approach in~\OPF{} problems.

\textit{Notation.} \(I_n\) denotes an \(n \times n\) identity matrix and \(\norm{\cdot}_{\ast}\) denotes the dual norm of \(\norm{\cdot}\). \(\Phi(\cdot)\) denotes the CDF of a 1-dimensional standard Gaussian.

\section{Convexity of \(\mathcal{Z}\)}
\label{sec:cvx}

We study the convexity of $\mathcal{Z}$. First, we review the definition of Value-at-Risk (\VaR)~\cite{rockafellar-1999-optim-condit}.
\begin{definition}
Let \(\tilde{X}\) be a random variable with distribution \(\Prob_{\tilde{X}}\). Then, the
\((1 - \epsilon)\)-\VaR{} of $\tilde{X}$ is defined as
\begin{align*}
  \VaR_{(1 - \epsilon)}(\tilde{X}) := \inf \Big\{ x \colon \Prob_{\tilde{X}} \big[ \tilde{X} \leq x \big] \geq 1 - \epsilon \Big\}.
\end{align*}
\end{definition}
Second, we specify the configurations of \(\mathcal{P}\).
\begin{assumption} \label{assump-wasser}
The Wasserstein ball $\mathcal{P}$ is such that
\begin{enumerate*}[label=(\alph*)]
\item the reference distribution \(\Prob\) is a multivariate Gaussian distribution \(\mathcal{N}(\mu, \Sigma)\) with \(\Sigma \succ 0\), and
\item the norm \(\norm{\cdot}\) in \(d_W\) is an ellipsoidal norm with regard to \(\Sigma^{1/2}\), \ie{}, \(\norm{\cdot} = \norm{\Sigma^{-1/2}(\cdot)}_2\) (or equivalently, \(\norm{\cdot}_{\ast} = \norm{\Sigma^{1/2}(\cdot)}_2\)).
\end{enumerate*}
\end{assumption}
Assumption~\ref{assump-wasser} facilitates the convex representation of \(\mathcal{Z}\) in Theorem~\ref{thm:convex} and the inner approximations derived in Section~\ref{sec:conic-approx}. In particular, Assumption 1(a) is motivated by the OPF problem, in which the forecast errors of renewable energy are usually modeled by a Gaussian distribution or other log-concave alternatives~\cite{lubin-2015-two-sided}. Consequently, the potential misspecification by \(\Prob\) can be restored by the robustness of (\textbf{2DRC}). In addition, Assumption 1(b) can be made without much loss of modeling power because all Wasserstein distances are equivalent.

We now establish a convex representation of \(\mathcal{Z}\).
\begin{lemma} \label{lemma:convex}
Suppose that \(\epsilon \in (0, 1/2)\) and Assumption~\ref{assump-wasser} holds. Define
\begin{align*}
  \mathcal{Z}_0 := \Set{ (\ell, u) \in \reals^2 \colon
  \inf_{\QProb \in \mathcal{P}_0} \QProb \left[ \ell \leq \tilde{\xi} \leq u \right] \geq 1 - \epsilon
  },
\end{align*}
where \(\mathcal{P}_0\) is a Wasserstein ball centered around the standard Gaussian distribution \(\Prob_0\) with the radius \(\delta\). Then, for any \(x \neq 0\), \((x, \ell, u) \in \mathcal{Z}\) if and only if \(\left(\frac{\ell - x^{\top}\mu}{\|x\|_*}, \frac{u - x^{\top}\mu}{\|x\|_*}\right) \in \mathcal{Z}_0\).
\end{lemma}
\begin{proof}
First, Proposition~\(1\) in~\cite{shen-2021-convex-chanc} implies that \((\ell, u) \in \mathcal{Z}_0\) if and only if it satisfies
\begin{gather}
f_0(\ell, u) \geq \delta \label{eq:deter-reform-Z0-1}\\
\text{and} \quad \Prob_0 \left[\ell \leq \tilde{\zeta}_0 \leq u \right] \geq 1 - \epsilon, \label{eq:deter-reform-Z0-2}
\end{gather}
where \(\tilde{\zeta}_0\) is the standard Gaussian random variable, \(\phi(\ell, u, \zeta) := \min \set{\zeta - \ell, u - \zeta}\), and 
\begin{gather*}
f_0(\ell, u) := \int\limits_0^{\VaR_{\epsilon}[ \phi(\ell, u, \tilde{\zeta}_0) ]} \Big(
\Prob_0 \left[  \phi(\ell, u, \tilde{\zeta}_0) \geq t \right] - ( 1 - \epsilon)\Big) \dLeb{t}.
\end{gather*}
Likewise, the same proposition implies that \((x, \ell, u) \in \mathcal{Z}\) if and only if it satisfies
\begin{gather*}
f(x, \ell, u) \geq \delta \\
\text{and} \quad \Prob\left[\ell \leq x^{\top} \tilde{\zeta} \leq u \right] \geq 1 - \epsilon,
\end{gather*}
where \(\tilde{\zeta}\) is a random variable with distribution \(\Prob\), \(\tilde{\psi} := \phi(\ell, u, x^{\top} \tilde{\zeta}) / \norm{x}_{\ast}\), and
\begin{align*}
f(x, \ell, u)
& := \int\limits_0^{\VaR_{\epsilon}(\tilde{\psi})}
\Big(\Prob \left[ \tilde{\psi} \geq t \right] - ( 1 - \epsilon) \Big) \dLeb{t}.
\end{align*}

Second, pick any \((x, \ell, u) \in \mathcal{Z}\) with \(x \neq 0\). By definition, we recast \(\tilde{\psi}\) as
\begin{align*}
\hspace{-0.4cm} & \ \frac{\min\{x^{\top} \tilde{\zeta} - \ell, u - x^{\top} \tilde{\zeta}\}}{\|x\|_*} \\
\hspace{-0.4cm} = & \ \min\left\{ \frac{x^{\top} (\tilde{\zeta} - \mu)}{\|x\|_*} - \frac{\ell - x^{\top} \mu}{\|x\|_*}, \frac{u - x^{\top} \mu}{\|x\|_*} - \frac{x^{\top} (\tilde{\zeta} - \mu)}{\|x\|_*} \right\}.
\end{align*}
Since \(x^{\top} (\tilde{\zeta} - \mu)/\|x\|_*\) is Gaussian, \(\mathbb{E}[x^{\top} (\tilde{\zeta} - \mu)/\|x\|_*] = 0\), and \(\text{Var}[x^{\top} (\tilde{\zeta} - \mu)/\|x\|_*] = 1\), \(x^{\top} (\tilde{\zeta} - \mu)/\|x\|_*\) is a standard Gaussian random variable and \(\tilde{\psi}\) follows the same distribution as \(\phi\big((\ell - x^{\top} \mu)/\|x\|_*, (u - x^{\top} \mu)/\|x\|_*, \tilde{\zeta}_0\big)\). It follows that
\begin{align*}
f_0\left( \frac{\ell - x^{\top} \mu}{\|x\|_*}, \frac{u - x^{\top} \mu}{\|x\|_*} \right) = f(x, \ell, u) \geq \delta.
\end{align*}
Similarly, we have
\begin{align*}
& \Prob_0 \left[\frac{\ell - x^{\top} \mu}{\|x\|_*} \leq \tilde{\zeta}_0 \leq \frac{u - x^{\top} \mu}{\|x\|_*} \right] \\
= \ & \Prob \left[\frac{\ell - x^{\top} \mu}{\|x\|_*} \leq \frac{x^{\top}(\tilde{\zeta} - \mu)}{\|x\|_*} \leq \frac{u - x^{\top} \mu}{\|x\|_*} \right] \\
= \ & \Prob \left[\ell \leq x^{\top}\tilde{\zeta} \leq u \right] \geq 1 - \epsilon,
\end{align*}
which yields that \(\left(\frac{\ell - x^{\top}\mu}{\|x\|_*}, \frac{u - x^{\top}\mu}{\|x\|_*}\right) \in \mathcal{Z}_0\).

Third, pick any \((x, \ell, u) \in \mathbb{R}^{n+2}\) with \(\left(\frac{\ell - x^{\top}\mu}{\|x\|_*}, \frac{u - x^{\top}\mu}{\|x\|_*}\right) \in \mathcal{Z}_0\) and \(x \neq 0\). Then, same arguments yield that
\begin{align*}
f(x, \ell, u) = f_0\bigg( & \frac{\ell - x^{\top} \mu}{\|x\|_*}, \frac{u - x^{\top} \mu}{\|x\|_*} \bigg) \geq \delta \\
\text{and} \quad \Prob \left[\ell \leq x^{\top}\tilde{\zeta} \leq u \right] = & \ \Prob_0 \left[\frac{\ell - x^{\top} \mu}{\|x\|_*} \leq \tilde{\zeta}_0 \leq \frac{u - x^{\top} \mu}{\|x\|_*} \right] \\
\geq & \ 1 - \epsilon.
\end{align*}
It follows that \((x, \ell, u) \in \mathcal{Z}\) and this completes the proof.
\end{proof}

\begin{theorem} \label{thm:convex}
Suppose that \(\epsilon \in (0, 1/2)\) and Assumption~\ref{assump-wasser} holds. Define
\begin{gather}
g_{\epsilon}(\ell, u) := \int\limits_0^{+\infty} \big[\Phi(u - t) - \Phi(\ell + t) - (1 - \epsilon) \big]^+ \dLeb{t} \nonumber \\
\text{and} \quad \mathcal{Z}_1 := \cl{} \set{ (\ell, u, s) \colon s > 0, (\ell / s, u / s)\in \mathcal{Z}_0 }, \label{eq:cvx-def-Z1}
\end{gather}
where \(\cl{}\) denotes the closure operator. Then,
\begin{gather*}
\mathcal{Z}_0 = \Set{(\ell, u) \colon \delta \leq g_{\epsilon}(\ell, u)}.
\end{gather*}
In addition, \((x, \ell, u) \in \mathcal{Z}\) if and only if there exists an \(s \geq \|x\|_*\) such that \((\ell - x^{\top}\mu, u - x^{\top}\mu, s) \in \mathcal{Z}_1\). Finally, both \(\mathcal{Z}_0\) and \(\mathcal{Z}\) are convex and closed.
\end{theorem}
\begin{proof}
First, Theorem~\(8\) in~\cite{shen-2021-convex-chanc} and Theorem~\(4.39\) in~\cite{shapiro-2009-lectur} imply that~\eqref{eq:deter-reform-Z0-1}
and~\eqref{eq:deter-reform-Z0-2} produce a convex and closed feasible region, i.e., \(\mathcal{Z}_0\) is convex and closed.

Second, we represent \(f_0(\ell, u)\) as
\begin{align*}
& \, \int\limits_0^{\VaR_{\epsilon}[ \phi(\ell, u, \tilde{\zeta}_0)] }
\Big(\Prob_0 \left[  \phi(\ell, u, \tilde{\zeta}_0) \geq t \right] - ( 1 - \epsilon)\Big) \dLeb{t} \\
= \, & \, \int\limits_0^{+\infty}
\Big(\Prob_0 \left[  \phi(\ell, u, \tilde{\zeta}_0) \geq t \right] - ( 1 - \epsilon)\Big)^+ \dLeb{t} \\
= \, & \, \int\limits_0^{+\infty}
\Big[\Phi(u - t) - \Phi(\ell + t) - ( 1 - \epsilon) \Big]^+ \dLeb{t},
\end{align*}
where the first equality is because the integrand is monotonically decreasing in \(t\) and the second equality is by definition of the function \(\phi\). Since \(\delta > 0\), constraint~\eqref{eq:deter-reform-Z0-1} implies that there exists a \(t \geq 0\) such that \(\Phi(u - t) - \Phi(\ell + t) > 1 - \epsilon\), or equivalently, \(\Prob_0[\ell + t \leq \tilde{\zeta}_0 \leq u - t] > 1 - \epsilon\), which implies constraint~\eqref{eq:deter-reform-Z0-2}. Hence, \(\mathcal{Z}_0 = \Set{(\ell, u) \colon \delta \leq g_{\epsilon}(\ell, u)}\).

Third, \(\mathcal{Z}_1\) is convex and closed because it is the conic hull of \(\mathcal{Z}_0\). Hence, to prove that \(\mathcal{Z}\) is convex and closed, it remains to show that \((x, \ell, u) \in \mathcal{Z}\) if and only if there exists an \(s \geq \|x\|_*\) such that \((\ell - x^{\top}\mu, u - x^{\top}\mu, s) \in \mathcal{Z}_1\). To this end, we discuss the following two cases.
\begin{enumerate}
\item Suppose that \(x = 0\). For any \((0, \ell, u) \in \mathcal{Z}\), we have \(\ell \leq 0 \leq u\) because otherwise \(\Prob[\ell \leq 0 \leq u] < 1/2 < 1 - \epsilon\), violating the assumption that \((0, \ell, u) \in \mathcal{Z}\). Then, \(s := 1/n\) for a sufficiently large integer \(n\) ensures that \((\ell/s, u/s) \in \mathcal{Z}_0\) and so \((\ell, u, s) \in \mathcal{Z}_1\). On the contrary, for any \((0, \ell, u) \in \reals^{n+2}\) such that there exists an \(s \geq 0\) with \((\ell, u, s) \in \mathcal{Z}_1\), by definition of \(\mathcal{Z}_1\) there exists a sequence \(\{(\ell_n, u_n, s_n)\}_{n=1}^{\infty}\) converging to \((\ell, u, s)\) such that \(s_n > 0\) and \(g_{\epsilon}(\ell_n/s_n, u_n/s_n) \geq \delta\) for all \(n\). Then, \(\ell_n < 0\) and \(u_n > 0\) for all \(n\) because otherwise \(g_{\epsilon}(\ell_n/s_n, u_n/s_n) = 0 < \delta\). Driving \(n\) to infinity yields that \(\ell \leq 0\) and \(u \geq 0\). Hence, \((0, \ell, u) \in \mathcal{Z}\).
\item Suppose that \(x \neq 0\). Pick any \((x, \ell, u) \in \mathcal{Z}\), then Lemma~\ref{lemma:convex} implies that \(\left(\frac{\ell - x^{\top} \mu}{\|x\|_*}, \frac{u - x^{\top} \mu}{\|x\|_*} \right) \in \mathcal{Z}_0\). Hence, \(s := \|x\|_* > 0\) ensures that \((\ell - x^{\top} \mu, u - x^{\top} \mu, s) \in \mathcal{Z}_1\). On the contrary, pick any \((x, \ell, u) \in \reals^{n+2}\) such that \(x \neq 0\) and there exists an \(s \geq \|x\|_* > 0\) with \((\ell - x^{\top}\mu, u - x^{\top}\mu, s) \in \mathcal{Z}_1\). By definition of \(\mathcal{Z}_1\), there exists a sequence \(\{(\ell_n, u_n, s_n)\}_{n=1}^{\infty}\) converging to \((\ell - x^{\top}\mu, u - x^{\top}\mu, s)\) such that \(s_n > 0\) and \(g_{\epsilon}(\ell_n/s_n, u_n/s_n) \geq \delta\) for all \(n\). Then, 
\begin{align*}
g_{\epsilon}\left(\frac{\ell - x^{\top}\mu}{\|x\|_*}, \frac{u - x^{\top}\mu}{\|x\|_*}\right) \geq & \  g_{\epsilon}\left(\frac{\ell - x^{\top}\mu}{s}, \frac{u - x^{\top}\mu}{s}\right) \\
= & \ \lim_{n \rightarrow \infty} g_{\epsilon}\left(\frac{\ell_n}{s_n}, \frac{u_n}{s_n}\right) \geq \delta,
\end{align*}
where the first inequality is because the function \(g_{\epsilon}(\ell, u)\) is nonincreasing in \(\ell\) and nondecreasing in \(u\), and the equality is due to the dominated convergence theorem. It follows that \(\left(\frac{\ell - x^{\top} \mu}{\|x\|_*}, \frac{u - x^{\top} \mu}{\|x\|_*} \right) \in \mathcal{Z}_0\) and so \((x, \ell, u) \in \mathcal{Z}\) by Lemma~\ref{lemma:convex}. This completes the proof.
\end{enumerate}
\end{proof}

\section{Tight Conic Approximation of \(\mathcal{Z}\)}
\label{sec:conic-approx}
Although Theorem~\ref{thm:convex} produces a convex representation of \(\mathcal{Z}\), it is not computable because \(g_{\epsilon}(\ell, u)\) is defined by an integration. In this section, we derive an \emph{inner} approximation of \(\mathcal{Z}\) from that of \(\mathcal{Z}_0\). The basic idea was proposed by~\cite{lubin-2015-two-sided} to derive outer approximations for chance constraints.

\subsection{Polyhedral inner approximation of \(\mathcal{Z}_0\)}
\label{sec:poly-approx-Z0}

To illustrate the basic idea, we define the \(\delta\)-level set of
\(g_{\epsilon}(\ell, u)\):
\begin{align*}
  \mathcal{C}_{\delta} := \Set{ (\ell, u) \in \reals^2 \colon g_{\epsilon}(\ell, u) = \delta }.
\end{align*}
The set \(C_{\delta} \subseteq \reals_- \times \reals_+\) because \(\epsilon < 1/2\). We plot \(C_{\delta}\) with fixed \(\epsilon = 0.1, \delta = 0.05\) and varying \(\delta\) in Figure~\ref{fig:Z0-approx}, from which we
observe that (i) \(C_{\delta}\) is convex and (ii) a polyhedral inner approximation of \(C_{\delta}\) can be constructed based on a set of points on
\(\mathcal{C}_{\delta}\).
\begin{figure}[!htbp]
\centering
\includegraphics[width=0.8\linewidth]{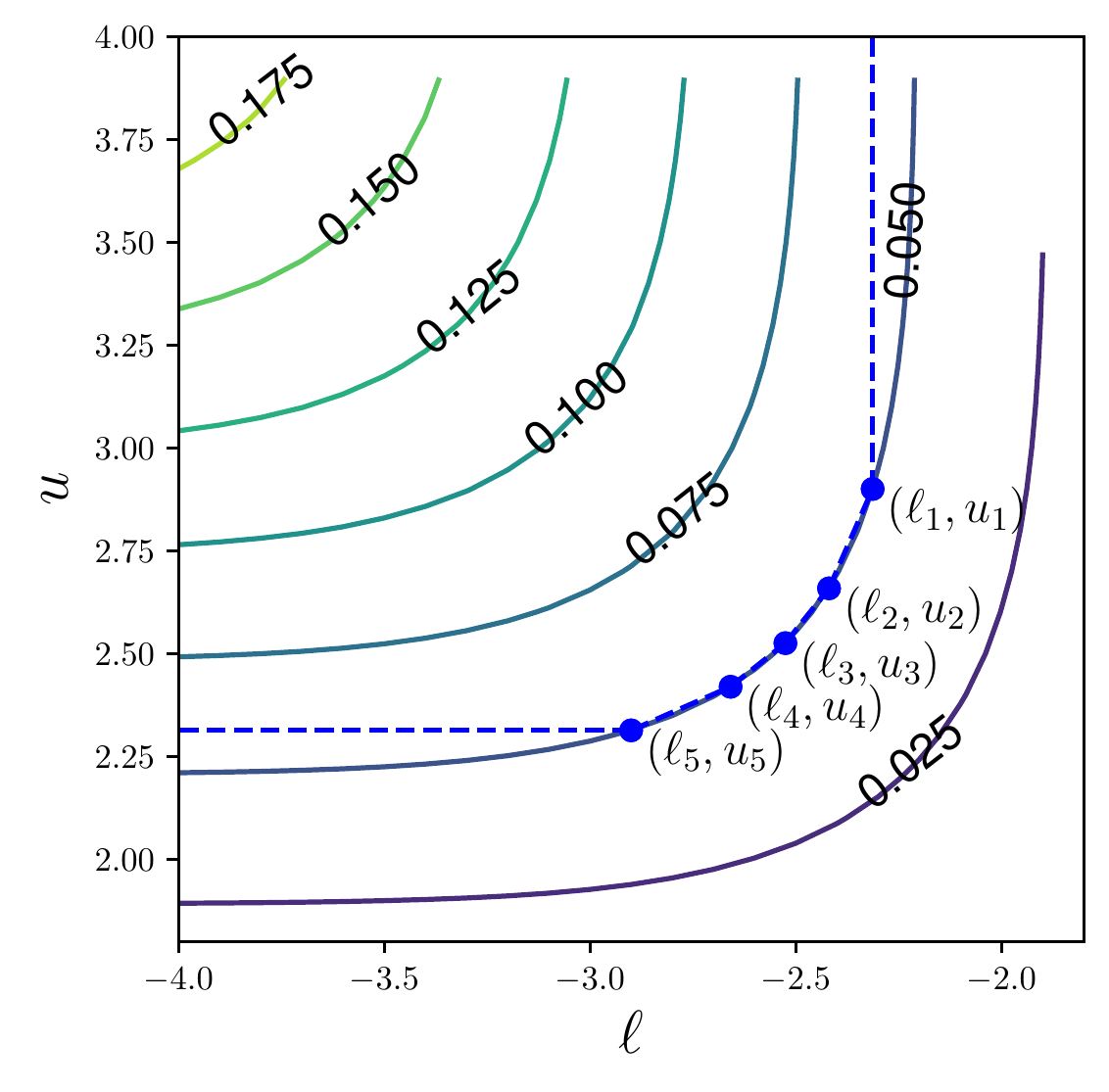}
\caption{Contour of \(g_{\epsilon}(\ell, u)\) with varying \(\delta\) and a polyhedral inner approximation}
\label{fig:Z0-approx}
\end{figure}
We now formalize this idea.
\begin{definition}
Define \(\hat{\mathcal{Z}}^N_0\) as the inner approximation of \(\mathcal{Z}_0\) (i.e., \(\mathcal{Z}^N_0 \subseteq \mathcal{Z}_0\)) obtained from the \(N\) points
\[ \set{(\ell_1, u_1), \ldots, (\ell_i, u_i), \ldots, (\ell_N, u_N)}\]
on \(C_{\delta}\) with decreasing \(\ell_i\)'s. Then, 
\[\hat{\mathcal{Z}}^N_0 := \set{(\ell, u) \colon \eqref{eq:approx-ZN0-p1},~\eqref{eq:approx-ZN0-pi},~\eqref{eq:approx-ZN0-pN}},\]
where
\begin{align}
  & \ell \leq \ell_1, \label{eq:approx-ZN0-p1}\\
  & \forall i \in [N-1]: \nonumber{} \\
  & (u - u_i) (\ell_i - \ell_{i + 1}) \geq (u_i - u_{i+1}) (\ell - \ell_i), \label{eq:approx-ZN0-pi} \\
  & u \geq u_N. \label{eq:approx-ZN0-pN}
\end{align}
\end{definition}
Specifically, the boundary of~\eqref{eq:approx-ZN0-p1}
(resp.~\eqref{eq:approx-ZN0-pN}) is the vertical (resp.~horizontal) ray emitting
from \((\ell_1, u_1)\) (resp.~\((\ell_N, u_N)\)) and the boundaries
of~\eqref{eq:approx-ZN0-pi} are the line segments connecting \((\ell_i, u_i)\)
Accordingly, we obtain the following conic inner approximation of \(\mathcal{Z}\) by Theorem~\ref{thm:convex}:
\begin{align*}
  \hat{\mathcal{Z}}^N := \Set{ (x, \ell, u) \colon
  \begin{aligned}
  & \exists s \in \reals: \norm{x}_{\ast} \leq s, \\
  & \big((\ell - x^{\top} \mu) / s, (u - x^{\top} \mu) / s\big) \in \hat{\mathcal{Z}}^N_0
  \end{aligned}
  },
\end{align*}
where the last constraint in \(\hat{\mathcal{Z}}^N\) can be recast as the following linear inequalities:
\begin{align*}
  & \left\{
    \begin{aligned}
    & \ell - x^{\top} \mu \leq \ell_1 s, \\
    & u - x^{\top} \mu \geq u_N s, \\
    & \forall i \in [N - 1]: \\
    & \left(\frac{\ell_i - \ell_{i+1}}{u_i - u_{i+1}}\right) \big(u - x^{\top} \mu - u_i s\big) \geq \ell - x^{\top} \mu - \ell_i s.
    \end{aligned}\right.
\end{align*}
\(\hat{\mathcal{Z}}^N\) can be directly computed by commercial solvers, and it inherits the approximation guarantee of
\(\hat{\mathcal{Z}}^N_0\), which we analyze in
Sections~\ref{sec:approx-pieces-i}--\ref{sec:approx-pieces-1N}.

\subsection{Approximation error induced by~\eqref{eq:approx-ZN0-pi}}
\label{sec:approx-pieces-i}

We first quantify the error of approximating a concave function \(h\) by an affine function \(\bar{h}\) from above.
\begin{lemma}\label{lem:concave-peaky-approx}
Suppose that \(h(\lambda) \colon [0, 1] \to \reals_+\) is a positive,
subdifferentiable, and strictly concave function. Define
\begin{align*}
\hat{h}(\lambda) := \min \set{ h(0) + \partial h(0) \lambda, h(1) + \partial h(1) (\lambda - 1)},
\end{align*}
where \(\partial h(\lambda)\) is a subgradient of \(h\) at
\(\lambda\),
\begin{align}
  \tau := \frac{\hat{h}(\lambda_{\ast})}{ \lambda_{\ast} h(1) + (1 - \lambda_{\ast}) h(0)} \geq 1, \label{eq:approx-construct-tau}
\end{align}
where
\(\lambda_{\ast} := \frac{h(1) - h(0) - \partial h(1)}{\partial h(0) - \partial h(1)} \in [0, 1]\), and
\begin{align*}
  \bar{h}(\lambda) := \tau \cdot  \big(h(1) - h(0)\big)\lambda + \tau \cdot h(0).
\end{align*}
Then, \(\frac{1}{\tau}\bar{h}(\lambda) \leq h(\lambda) \leq  \bar{h}(\lambda)\) for all \(\lambda \in [0,1]\).
\end{lemma}

\begin{proof}
By concavity of \(h\), \(h(\lambda) \leq \hat{h}(\lambda)\). First, we show \(h(\lambda) \leq \bar{h}(\lambda)\). Observe that:
(i) \(\overline{h}(0) = \tau h(0) \geq h(0) = \hat{h}(0)\), and
\(\overline{h}(\lambda_{\ast}) = \hat{h}(\lambda_{\ast})\), implying
\(\nabla \overline{h} \leq \partial h(0)\); (ii) \(\overline{h}(1) = \tau h(1) \geq h(1) = \hat{h}(1)\), implying
\(\nabla \overline{h} \geq \partial h(1)\). It follows that \(\overline{h}\) is a supporting
hyperplane of the hypograph of \(\hat{h}\). Thus, we have \(h(\lambda) \leq \hat{h}(\lambda) \leq \bar{h}(\lambda)\) for all \(\lambda \in [0, 1]\).
Second, \(\frac{1}{\tau}\bar{h}(\lambda) \leq h(\lambda)\) because \(h\) is concave.
\end{proof}
Now we quantify the approximation error induced by~\eqref{eq:approx-ZN0-pi}.
\begin{proposition}\label{prop:cvx-g-2points}
Suppose that \(\epsilon \in (0, 1/2), \delta > 0\), and
\((\ell_1, u_1), (\ell_2, u_2)\) are two points in
\(\mathcal{C}_{\delta}\). Let
\((\ell_{\lambda}, u_{\lambda})\) be their convex combination such that
\((\ell_{\lambda}, u_{\lambda}) := (1 -\lambda) (\ell_1, u_1) + \lambda (\ell_2, u_2)\) for  \(\lambda \in [0, 1]\) and define
\(
  s(\lambda) := g_{\epsilon}(\ell_{\lambda}, u_{\lambda}).
\)
Then, it holds that
\begin{enumerate}
\item \(\sqrt{s(\lambda)}\) is a positive, concave, and differentiable function over an open interval containing \([0, 1]\);
\item \(\sup_{\lambda \in [0, 1]} s(\lambda) \leq \tau_s^2 \delta\), where \(\tau_s\) is constructed from~\eqref{eq:approx-construct-tau} by replacing \(h(\lambda)\) with \(\sqrt{s(\lambda)}\).
\item \(1 \leq \tau_s \leq 1 + O(\norm{(\ell_1, u_1) - (\ell_2, u_2)}_1)\).
\end{enumerate}
\end{proposition}

\begin{proof}
Since \(\epsilon < 1/2\), the function
\(\varphi (\ell, u, t) := \Phi(u - t) - \Phi(\ell + t)\) is jointly concave on
\(\{(\ell, u, t) \colon \varphi(\ell, u, t) \geq (1 - \epsilon)\}\). Then, \(\sqrt{g_{\epsilon}(\ell, u)}\) is concave on
\(\set{(\ell, u) \colon \Prob_0 [ \ell \leq \tilde{\zeta}_0 \leq u] \geq 1 - \epsilon}\) by Theorem~\(2\) in~\cite{gupta-1980-brunn-minkow}.
Thus, \(\sqrt{s(\lambda)}\) is positive and concave over an open
interval containing \([0, 1]\). By the  Leibnitz integration rule,
\(s(\lambda)\) is differentiable on \([0, 1]\) and 
\begin{align*}
  & \frac{d}{d \lambda} s(\lambda) \\
  = \;
  & \int\limits_0^{+\infty} \frac{d}{d \lambda} \Big( \Phi(u_{\lambda} - t) - \Phi(\ell_{\lambda} + t) - (1 - \epsilon)\Big)^+ \dLeb{t} \\
  = \;
  & \int\limits_0^{+\infty} \frac{1}{\sqrt{2\pi}} \left(
    e^{-\frac{(u_{\lambda} - t)^2}{2}} (u_2 - u_1) - e^{-\frac{(\ell_{\lambda} + t)^2}{2}} (\ell_2 - \ell_1) \right) \\
  & \phantom{\int\limits_0^{+\infty}}
    \cdot \Ind{\Phi(u_{\lambda} - t) - \Phi(\ell_{\lambda} + t) \geq (1 - \epsilon)}
    \dLeb{t}.
\end{align*}
Hence, \(\sqrt{s(\lambda)}\) is differentiable by the chain rule. By
Lemma~\ref{lem:concave-peaky-approx}, there exists a \(\tau_s > 0\) such that
\begin{align*}
  \sqrt{s(\lambda)}
  \leq \tau_s\big((1 - \lambda) \sqrt{s(0)} + \lambda \sqrt{s(1)} \big)
  = \tau_s \sqrt{\delta}.
\end{align*}
Then, for \(M_{1,2} := \max\set{\abs{u_1}, \abs{u_2}, \abs{\ell_1}, \abs{\ell_2}}\), we have
\begin{align*}
  & \abs[\Big]{\frac{d}{d \lambda} s(\lambda)} \\
  \leq \;
  & \int\limits_0^{+\infty} \frac{1}{\sqrt{2\pi}} \big( \abs{u_2 - u_1} + \abs{\ell_2 - \ell_1} \big) \\
  & \quad \cdot \Ind{\Phi(u_{\lambda} - t) - \Phi(\ell_{\lambda} + t) \geq (1 - \epsilon)} \dLeb{t} \\
  \leq \;
  & \left( \abs{u_2 - u_1} + \abs{\ell_2 - \ell_1} \right) \cdot \Leb{}\big(\big[0, \min\set{\abs{u_{\lambda}}, \abs{\ell_{\lambda}}}\big]\big) \\
  \leq \;
  & M_{1,2} \left( \abs{u_2 - u_1} + \abs{\ell_2 - \ell_1} \right),
\end{align*}
where \(\Leb{}(\cdot)\) denotes the Lebsgue measure and 
the second inequality is because
\(\epsilon \in (0, 1/2)\), implying that \(u_{\lambda} - t \geq 0 \text{ and
} \ell_\lambda + t \leq 0\), \ie{}, \(t \leq u_{\lambda}\) and \(t \leq -\ell_{\lambda}\).
Finally, for all \(\lambda \in [0, 1]\), we derive
\begin{align*}
\hat{s}(\lambda) := \; & \min \left\{
  \sqrt{s(0)} + \frac{d}{d \lambda} \sqrt{s(0)} \cdot \lambda, \right. \\
  & \quad \left. \sqrt{s(1)} + \frac{d}{d \lambda} \sqrt{s(1)} \cdot (\lambda - 1) \right\}\\
  \leq \;
  & \sqrt{\delta} + \min \left\{
    \abs[\Big]{\frac{d}{d \lambda} \sqrt{s(0)}},
    \abs[\Big]{\frac{d}{d \lambda} \sqrt{s(1)}} \right\} \\
  = \;
  & \sqrt{\delta} +  \frac{1}{2\sqrt{\delta}} \min \left\{
    \abs[\Big]{\frac{d}{d \lambda} s(0)},
    \abs[\Big]{\frac{d}{d \lambda} s(1)} \right\} \\
  \leq \;
  & \sqrt{\delta} + \frac{1}{2 \sqrt{\delta}} M_{1,2} \big( \abs{u_2 - u_1} + \abs{\ell_2 - \ell_1} \big),
\end{align*}
where we use the fact \(s(0) = s(1) = \delta\). Thus, by definition of \(\tau_s\) we have
\begin{align*}
  \tau_s = &  \frac{\hat{s}(\lambda_{\ast})}{ \lambda_{\ast} \sqrt{s(1)} + (1 - \lambda_{\ast}) \sqrt{s(0)}} \\
  \leq \;
  & \frac{1}{\sqrt{\delta}} \left( \sqrt{\delta} + \frac{1}{2 \sqrt{\delta}} M_{1,2} \big( \abs{u_2 - u_1} + \abs{\ell_2 - \ell_1} \big) \right) \\
  = \;
  & 1 + \frac{M_{1,2}}{2 \delta} \norm{(\ell_1, u_1) - (\ell_2, u_2)}_1.
\end{align*}
This completes the proof.
\end{proof}

\subsection{Approximation errors induced by~\eqref{eq:approx-ZN0-p1} and~\eqref{eq:approx-ZN0-pN}}
\label{sec:approx-pieces-1N}

We define the approximation errors induced by \(\ell \leq \ell_1\) and \(u \geq u_N\) as 
\(\mathbf{err}_1 := \sup_{u \geq u_1} g_{\epsilon}(\ell_1, u) - \delta\) and \(\mathbf{err}_N := \sup_{\ell \leq \ell_N} g_{\epsilon}(\ell, u_N) - \delta\), respectively. Since the function \(g_{\epsilon}(\ell, u)\) is nonincreasing in \(\ell\) and nondecreasing in \(u\), we have
\begin{align}
  \overline{g}(u_N) & = \sup_{\ell \leq \ell_N} g_{\epsilon}(\ell, u_N) = g_{\epsilon}(-\infty, u_N), \label{eq:approx-err-ZN0-pN}\\
  \text{and} \quad \underline{g}(\ell_1) & = \sup_{u \geq u_1} g_{\epsilon}(\ell_1, u) = g_{\epsilon}(\ell_1, \infty), \label{eq:approx-err-ZN0-p1}
\end{align}
where we define, for any \((\ell, u) \in \mathbb{R}_-\times\mathbb{R}_+\),
\begin{align*}
  \overline{g}(u) & := \int\limits_0^{+\infty} \left( \Phi(u - t) - (1 - \epsilon) \right)^+ \dLeb{t}, \\
  \underline{g}(\ell) & := \int\limits_0^{+\infty} \left( \epsilon - \Phi(\ell + t) \right)^+ \dLeb{t}.
\end{align*}

The next two propositions imply that if \(\ell_N\) (resp.~\(u_1\)) is sufficiently
small (resp.~large) then \(\mathbf{err}_N\) (resp.~\(\mathbf{err}_1\))
becomes arbitrarily small.

\begin{proposition}\label{prop:cvx-g-limit-1}
Suppose that \(\epsilon \in (0, 1/2)\) and \(\delta > 0\). Then, for a sequence
of points 
\(\set{(\ell_n, u_n), n \in \naturals} \subseteq \mathcal{C}_{\delta}\),
if \(\ell_n \searrow -\infty\) as \(n \to \infty\) then \(u_n \to u^{\ast}\) as \(n \to \infty\), where \(u^{\ast}\) is the solution of the equation  \(\overline{g}(u) = \delta\).
\end{proposition}

\begin{proof}
Since \(\ell_n \searrow -\infty\) and \((\ell_n, u_n) \in \mathcal{C}_{\delta}\),
\(u_n\) is decreasing in \(n\). Consider the sequence of functions
\(\set{g_n, n \in \naturals}\), where
\begin{align*}
  g_n(u) := \int\limits_0^{+\infty} \big( \Phi(u - t) - \Phi(\ell_n + t) - (1 - \epsilon) \big)^+ \dLeb{t}.
\end{align*}
Evidently, \(g_n\) is increasing, bounded from above by \(\overline{g}\), and continuous for all \(n\) by the dominated convergence
theorem. Take a \(\underline{u} > 0\) such that
\(\overline{g}(\underline{u}) < \delta\) and define a restricted domain \(\dom_g := [\underline{u}, u_1]\) for all
\(g_n\)'s and \(\overline{g}\). Since
\(g_n(\underline{u}) \leq \underline{g}(\underline{u}) < \delta\) for all \(n\) and
\(g_n(u_1) \geq g_1(u_1) = \delta\), the solution of equations
\(\set{u \colon g_n(u) = \delta} \subseteq \dom_g\) by the intermediate value
theorem. First, we show that \(g_n \to \overline{g}\) uniformly as
\(n \to \infty\) on \(\dom_g\). Notice that
\begin{align*}
  & \abs{g_n(u) - \overline{g}(u)} \\
  \leq \;
  & \int\limits_0^{+\infty} \Big| \big( \Phi(u - t) - \Phi(\ell_n + t) - (1 - \epsilon) \big)^+ \\
  & \phantom{\int\limits_0^{+\infty}} - \big( \Phi(u - t) - (1 - \epsilon) \big)^+ \Big| \dLeb{t} \\
  \leq \;
  & \int\limits_0^{+\infty} \Phi(\ell_n + t) \cdot \Ind{ \Phi(u - t) \geq (1 - \epsilon)} \dLeb{t} \\
  = \;
  & \int\limits_0^{+\infty} \Phi(\ell_n + t) \cdot \Ind{ t \leq u_1 - \Phi^{-1} (1 - \epsilon)} \dLeb{t}.
\end{align*}
For any \(u \in \dom_g\), the dominated convergence theorem implies that
\begin{align*}
  & \lim_{n \to \infty} \abs[\Big]{g_n(u) - \overline{g}(u)}  \\
  \leq \;
  & \int\limits_0^{+\infty} \lim_{n \to \infty} \Phi(\ell_n + t) \cdot \Ind{ t \leq u_1 - \Phi^{-1} (1 - \epsilon)} \dLeb{t} = 0.
\end{align*}
Due to the strict monotonicity of \(\overline{g}\) in \(u\), its inverse
function \((\overline{g})^{-1}\) is well defined. Furthermore, it is continuous
because \(\dom_g\) is compact. Second, we bound the distance between
\(u_n\) and \(u_{\ast}\). For any \(\varepsilon > 0\), there exists
\(N_{\varepsilon} \in \naturals\) such that
\begin{align*}
  n > N_{\varepsilon} \implies \sup_{u \in \dom_g} \abs{g_n(u) - \overline{g}(u)} < \varepsilon.
\end{align*}
Let \(u^{\ast}_n\) be the solution of \(g_n(u) = \delta\), then for all
\(n > N_{\varepsilon}\),
\begin{align*}
  u_{\ast} \leq u^{\ast}_n \leq (\overline{g})^{-1}(\delta + \varepsilon),
\end{align*}
where the first inequality is because \(g_n\) is monotone and
\(g_n(u^{\ast}_n) = \delta = \overline{g}(u^{\ast}) \geq g_n(u^{\ast})\), and the
second inequality is because \((\overline{g})^{-1}\) is monotone and
\(\overline{g}(u^{\ast}_n) \leq g_n(u^{\ast}_n) + \varepsilon\). We complete the proof by noting that
\begin{align*}
  \inf_{\varepsilon > 0} \sup_{n \geq N_{\varepsilon}} \abs{ u^{\ast}_n - u^{\ast}}
  \leq \inf_{\varepsilon > 0} \left( (\overline{g})^{-1}(\delta + \varepsilon) - u^{\ast} \right) = 0,
\end{align*}
where the last equality is because \((\overline{g})^{-1}\) is continuous.
\end{proof}

\begin{proposition}\label{prop:cvx-g-limit-2}
Suppose that \(\epsilon \in (0, 1/2)\) and \(\delta > 0\). Then, for a sequence
of points 
\(\set{(\ell_n, u_n), n \in \naturals} \subseteq \mathcal{C}_{\delta}\), if 
\(u_n \nearrow +\infty\) as \(n \to \infty\), then \(\ell_n \to \ell^{\ast}\) as \(n \to \infty\), where  \(\ell^{\ast}\) is the
solution of the equation \(\underline{g}(\ell) = \delta\). 
\end{proposition}

\begin{proof}
By \(\Phi(x) = 1 - \Phi(-x)\), for any \(t \in \reals\), we have
\begin{align*}
  \Phi(u - t) - \Phi(\ell + t)
  & = 1 - \Phi(-u + t) - (1 - \Phi(-\ell - t)) \\
  & = \Phi(-\ell - t) - \Phi(-u + t),
\end{align*}
therefore \(g_{\epsilon}(\ell, u) = g_{\epsilon}(-u, -\ell)\), \ie{}
\(\set{(-u_n, -\ell_n)} \subseteq \mathcal{C}_{\delta}\) is a sequence of points
on \(\mathcal{C}_{\delta}\) with \(-u_n \searrow -\infty\) as \(n \to \infty\). Then, Proposition~\ref{prop:cvx-g-limit-1} yields that \(-\ell_n \to - \ell^{\ast}\).
\end{proof}



\subsection{Approximation bound of \(\hat{\mathcal{Z}}^N_0\)}
\label{sec:approx-bound-summary}

We summarize the approximation bounds derived in
Sections~\ref{sec:approx-pieces-1N} and~\ref{sec:approx-pieces-i} as follows.

\begin{theorem}\label{thm:non-asymptotic-bd}
Let \(\bd \big(\hat{\mathcal{Z}}^N_0\big)\) be the boundary of \(\hat{\mathcal{Z}}^N_0\),
then we have
\begin{align}
\delta
& \leq \max_{(\ell, u) \in \bd{} \big(\hat{\mathcal{Z}}^N_0\big)} g_{\epsilon}(\ell, u) \nonumber\\
& \leq \max\set{ (1 + O(\Delta^N)) \cdot \delta, \overline{g}(u_N), \underline{g}(\ell_1)}, \label{eq:approx-final-bd}
\end{align}
where
\(\displaystyle \Delta^N := \max_{i \in [N-1]} \norm{(\ell_i, u_i) - (\ell_{i+1}, u_{i+1})}_1\).
Furthermore, we have
\begin{align}
  \lim_{\substack{\Delta^N \to 0 \\ \ell_n \searrow -\infty \\ u_n \nearrow +\infty}}
  \max_{(\ell, u) \in \bd{} \big(\hat{\mathcal{Z}}^N_0\big)} g_{\epsilon}(\ell, u) = \delta. \label{eq:approx-final-limit}
\end{align}
\end{theorem}

\begin{proof}
In~\eqref{eq:approx-final-bd}, the first inequality is by construction, and
the second inequality follows from
Proposition~\ref{prop:cvx-g-2points} and definitions~\eqref{eq:approx-err-ZN0-pN}--\eqref{eq:approx-err-ZN0-p1}. Finally, equality~\eqref{eq:approx-final-limit} follows
from Propositions~\ref{prop:cvx-g-2points},~\ref{prop:cvx-g-limit-1},
and~\ref{prop:cvx-g-limit-2}.
\end{proof}

\section{Numerical Experiments}\label{sec:exps}

We evaluate the approximation bound of \(\hat{\mathcal{Z}}^N_0\) in Section~\ref{sec:exps-approx} and conduct a case study on~\DCOPF{} in
Section~\ref{sec:case-study}. All experiments are implemented using the
\href{https://www.gurobi.com/documentation/9.1/quickstart_mac/cs_python.html}{Python
  API} of
\href{https://www.gurobi.com/products/gurobi-optimizer/prior-version-enhancements/}{Gurobi
  \(9.1.1\)} and conducted on a single node of the
\href{https://arc-ts.umich.edu/greatlakes/}{Great Lakes} cluster provided by
University of Michigan, which contains two \(3.0\)GHz
\href{https://arc-ts.umich.edu/greatlakes/configuration/}{Intel Xeon Gold 6154}
CPUs.

\subsection{Approximation bound of \(\hat{\mathcal{Z}}^N_0\)}\label{sec:exps-approx}
\begin{algorithm}[tbp]
  \caption{\(\hat{\mathcal{Z}}^N_0\) Construction}\label{alo:obtain-N-pnts}
  \SetKwInOut{Inputs}{Inputs}
  \SetKw{Return}{return}
  \Inputs{\(\epsilon \in (0, \frac{1}{2}), \delta > 0\), and an odd integer \(N \geq 3\).}
  Find \(u_0, \overline{\ell}\) such that
  \[ g_{\epsilon}(-u_0, u_0) = \delta \quad \text{and} \quad
     \underline{g}(\overline{\ell}) = \delta,
     \] \\
  Obtain \((N-1) / 2\) evenly spaced points \(\set{\ell_i}_{i=1}^{(N-1)/2}\) over the 
     interval \([-u_0, \overline{\ell}]\). \\
  \For{\(i = 1, 2, \ldots, (N-1) / 2\)}{
    Find \(u_i\) such that \(g(\ell_i, u_i) = \delta\).
  }
  Collect all points
  \[\mathcal{L} := \set{(\ell_i, u_i), (-u_i, -\ell_i)}_{i=1}^{(N-2)/2} \cup \set{(-u_0, u_0)}\] \\
  \Return{\(\mathcal{L}\).}
\end{algorithm}
Given an odd integer \(N \geq 3\), we construct \(\hat{\mathcal{Z}}^N_0\)
using Algorithm~\ref{alo:obtain-N-pnts} and report the approximation bound
\begin{align*}
  \text{Apx-Bd} := \max\set{ (1 + O(\Delta^N)) \cdot \delta, \overline{g}(u_N), \underline{g}(\ell_1)} / \delta
\end{align*}
of \(\hat{\mathcal{Z}}^N_0\) in Table~\ref{tab:exps-approx-bd} with respect to various values of \(\epsilon\),
\(\delta\), and \(N\) using Theorem~\ref{thm:non-asymptotic-bd}.
From this table, we observe that for fixed \(\epsilon\) and \(\delta\), Apx-Bd decreases as \(N\) increases. For
example, when \(\epsilon = 0.01, \delta = 0.1\), Apx-Bd
improves from \(1.012\) to \(1.002\) when \(N\) increases from \(3\) to \(29\).
Furthermore, the larger the Wasserstein radius \(\delta\) is, the better our approximation becomes. For example, when \(\epsilon = 0.01, N = 3\), Apx-Bd improves from \(1.114\) to \(1.012\) as \(\delta\) increases from \(0.01\) to \(0.10\). In
addition, the marginal improvement in Apx-Bd diminishes as \(N\)
increases. For example, when \(\epsilon = 0.01, \delta = 0.1\), the improvement in Apx-Bd is \(0.004\) as \(N\) increases from \(3\) to \(9\), while from \(N = 19\) to
\(N=29\) the improvement is less than \(0.001\).

\begin{table}[!htbp]
\centering
\caption{Approximation bound of \(\hat{\mathcal{Z}}^N_0\)}\label{tab:exps-approx-bd}
\vspace{-1em}
\subfloat[]{%
\begin{tabular}{lllr}
\toprule
\(\epsilon\) & \(\delta\) & \(N\) & Apx-Bd \\
\midrule
0.01         & 0.01       & 3     & 1.114 \\
             &            & 5     & 1.076 \\
             &            & 9     & 1.046 \\
             &            & 19    & 1.023 \\
             &            & 29    & 1.016 \\
\midrule
             & 0.05       & 3     & 1.023 \\
             &            & 5     & 1.016 \\
             &            & 9     & 1.010 \\
             &            & 19    & 1.006 \\
             &            & 29    & 1.004 \\
\midrule
             & 0.10       & 3     & 1.012 \\
             &            & 5     & 1.008 \\
             &            & 9     & 1.005 \\
             &            & 19    & 1.002 \\
             &            & 29    & 1.002 \\
\bottomrule
\end{tabular}
}
\subfloat[]{%
\begin{tabular}{lllr}
\toprule
\(\epsilon\) & \(\delta\) & \(N\) & Apx-Bd \\
\midrule
0.05         & 0.01       & 3     & 1.537 \\
             &            & 5     & 1.350 \\
             &            & 9     & 1.207 \\
             &            & 19    & 1.102 \\
             &            & 29    & 1.068 \\
\midrule
             & 0.05       & 3     & 1.137 \\
             &            & 5     & 1.091 \\
             &            & 9     & 1.055 \\
             &            & 19    & 1.028 \\
             &            & 29    & 1.019 \\
\midrule
             & 0.10       & 3     & 1.068 \\
             &            & 5     & 1.046 \\
             &            & 9     & 1.027 \\
             &            & 19    & 1.013 \\
             &            & 29    & 1.009 \\
\bottomrule
\end{tabular}}
\end{table}

\subsection{A case study on \DCOPF{}}\label{sec:case-study}

In a transmission grid, the \OPF{} problem seeks to find a minimum-cost plan for power generation and transmission so that all electricity loads are satisfied and all system safety conditions, including the power generation limits and the transmission capacity limits, are respected. When uncertain renewable energy (e.g., wind power) is incorporated, chance-constrained \OPF{}~\cite{bienstock-2014-chanc-const} is a natural approach to keeping the system safe with high probability. We first introduce some notation. \(\mathcal{B}\) and \(\mathcal{G}\) denote the index sets of buses and thermal
generators, respectively. For two buses \(i, j \in \mathcal{B}\), \((i, j)\) denotes the
directed branch from \(i\) to \(j\) and \(\mathcal{E}\) represents the set of
all branches. We use subscripts \(i, j \in B\), \(g \in \mathcal{G}\), or \((i, j) \in \mathcal{E}\) to denote a quantity
related a specific bus, generator, or branch. For example, \(P^{\text{max}}_g\) and \(P^{\text{min}}_g\) denote the maximum and minimum power generation capacity of thermal unit \(g \in \mathcal{G}\), respectively. For each branch
\((i, j) \in \mathcal{E}\), \(f_{ij}\) and \(f^{\text{max}}_{ij}\) 
denote the power flow on branch \((i, j)\) and its maximum capacity,
respectively. In addition, \(\overline{p}_g \in \reals\) denotes the
amount of power generation of each thermal generator
\(g \in \mathcal{G}\), and \(d_i\) and \(\overline{\theta}_i\) denote the
electricity load and voltage phase angle at each bus \(i \in \mathcal{B}\), respectively.

We model the uncertain power output of each
renewable source \(i \in \mathcal{B}\) as \(\mu_i + \tilde{\xi}_i\), where \(\mu_i\)
represents the forecast amount of power generation and \(\tilde{\xi}_i\) is a zero-mean random variable representing the forecast error. In
response to the uncertain fluctuation in renewable output, we adjust the power
outputs of the thermal units using Automatic Generation Control, \ie{},
\(\tilde{p}_g := \overline{p}_g - \alpha_g \cdot \tilde{\xi}_{\text{tot}}\) for all \(g \in \mathcal{G}\),
where decision variable \(\alpha_g\) is called the participation factor of \(g\), and it represents the percentage of the total forecast error
\(\xitot{} := \sum_{i \in \mathcal{B}} \tilde{\xi}_i\) compensated by generator
\(g\). The chance-constrained \OPF{} with (\textbf{2DRC}) is formulated as follows.
\begin{subequations}\label{eq:exps-ts-drc}
\begin{align}
  \min~
  & \sum_{g \in \mathcal{G}} c_g(\overline{p}_g) \label{eq:exps-obj}\\
  \st{}~
  & \sum_{i \in \mathcal{G}} \alpha_i = 1, \alpha \geq 0, \overline{p}_g \geq 0, \label{eq:exps-alph-domain}\\
  & \sum_{i \in \mathcal{B}} (\overline{p}_i + \mu_i + d_i) = 0, \label{eq:exps-power-bal}\\
  & B \overline{\theta} = \overline{p} + \mu + d, \label{eq:exps-pflow-bal}\\
  & \inf_{\QProb \in \mathcal{P}_g} \QProb{} \left(
    P^{\text{min}}_g \leq \overline{p}_g - \tilde{\xi}_{\text{tot}} \; \alpha_i \leq P_g^{\text{max}} \right) \geq 1 - \epsilon_g, \forall g \in \mathcal{G}, \label{eq:exps-pgen-bd}\\
  & \inf_{\QProb \in \mathcal{P}_{ij}} \QProb{} \left( f^{\text{min}}_{ij} \leq \beta_{ij} (\overline{\theta}_i - \overline{\theta}_j) +
    \left[ \breve{\mathcal{B}} (\tilde{\xi} - \tilde{\xi}_{\text{tot}} \alpha) \right]_i - \right. \nonumber{}\\
  & \left. \qquad \left[ \breve{\mathcal{B}} (\tilde{\xi} - \tilde{\xi}_{\text{tot}} \alpha)\right]_j \leq f^{\text{max}}_{ij} \right) \geq 1 - \epsilon_b, \forall (i, j) \in \mathcal{E}, \label{eq:exps-flow-bd}
\end{align}
\end{subequations}
where \(c_g(\cdot) \colon \reals \to \reals_+\) is a quadratic function
representing the fuel cost of thermal generator \(g \in \mathcal{G}\),
\(d\) denotes the vector of electricity loads, \(\overline{p}_g\) denotes the vector of power outputs, \(\beta_{ij}\) denotes the line susceptance of \((i, j) \in \mathcal{E}\), and matrices \(B\) and
\(\breve{B}\) denote the weighted Laplacian matrix and its pseudo-inverse, respectively (see
Equation~\((1.5)\) and~\((2.5)\) in~\cite{bienstock-2014-chanc-const}). For each \(g \in \mathcal{G}\) (resp. \((i,j) \in \mathcal{E}\)), \(\mathcal{P}_g\) (resp.
\(\mathcal{P}_{ij}\)) represents a Wasserstein ball centered around a
Gaussian distribution with empirical mean and covariance matrix. Finally, 
\(1-\epsilon_g\) and \(1-\epsilon_b\) are risk thresholds for the power generation limit and transmission capacity limit constraints, respectively.

We demonstrate the out-of-sample (OOS) performance of the (\textbf{2DRC}) formulation~\eqref{eq:exps-ts-drc} on a modified IEEE \(118\)-bus system, where we follow~\cite{bienstock-2014-chanc-const} to adjust the capacities of branches, and we place four wind farms at buses \(2, 7, 43\), and \(86\). The true distribution of the wind power output is assumed to be
Weibull with scale parameter \(1.0\) and shape parameters \(1.2, 3.5, 0.5, 4.0\), respectively. For a fixed
solution of~\eqref{eq:exps-ts-drc}, its OOS performance refers to the probability of violating the system safety conditions~\eqref{eq:exps-pgen-bd}--\eqref{eq:exps-flow-bd} under the true distribution. Specifically, we draw \(10,000\) samples from the true distribution to obtain an empirical estimate of OOS. In this
experiment, we first obtain \(M \in \set{5, 10, 100, 200, 500}\)
training data from the true distribution and construct the Wasserstein balls \(\mathcal{P}_g\) and \(\mathcal{P}_{ij}\) using empirical mean and covariance with respect to the five different training data sizes. Then, with \(\epsilon_g = \epsilon_b = 0.05\) and Wasserstein radii
\(\delta \in \set{0.01, 0.05, 0.1}\), we generate \(5\) random instances for all parameter settings, each of which is solved using formulation~\eqref{eq:exps-ts-drc} with (\textbf{2DRC}) and with (\textbf{CC}), respectively. In addition, we estimate the
average OOS, as well as its \(95\%\) confidence interval, for both solutions and report the results in Figs.~\ref{fig:exps-non-Gauss-oos-vs-nsmpls}--\ref{fig:exps-non-Gauss-oos-vs-delta}. From
Fig.~\ref{fig:exps-non-Gauss-oos-vs-nsmpls}, we observe that as the training data size \(M\) increases the OOS of both models improve. Nevertheless, (\textbf{2DRC}) achieves an OOS of at least \(95\%\) with as few as \(10\) training data, while (\textbf{CC}) fails to
achieve the target threshold even with \(500\) training data. From Fig.~\ref{fig:exps-non-Gauss-oos-vs-delta}, we notice that the OOS of (\textbf{2DRC}) exceeds the target risk threshold once \(\delta\) reaches \(0.05\), and it keeps improving as \(\delta\) increases further.
\begin{figure}[!htbp]
\centering
\subfloat[\(\delta = 0.05\)\label{fig:exps-non-Gauss-oos-vs-nsmpls}]{%
  \includegraphics[width=0.8\linewidth]{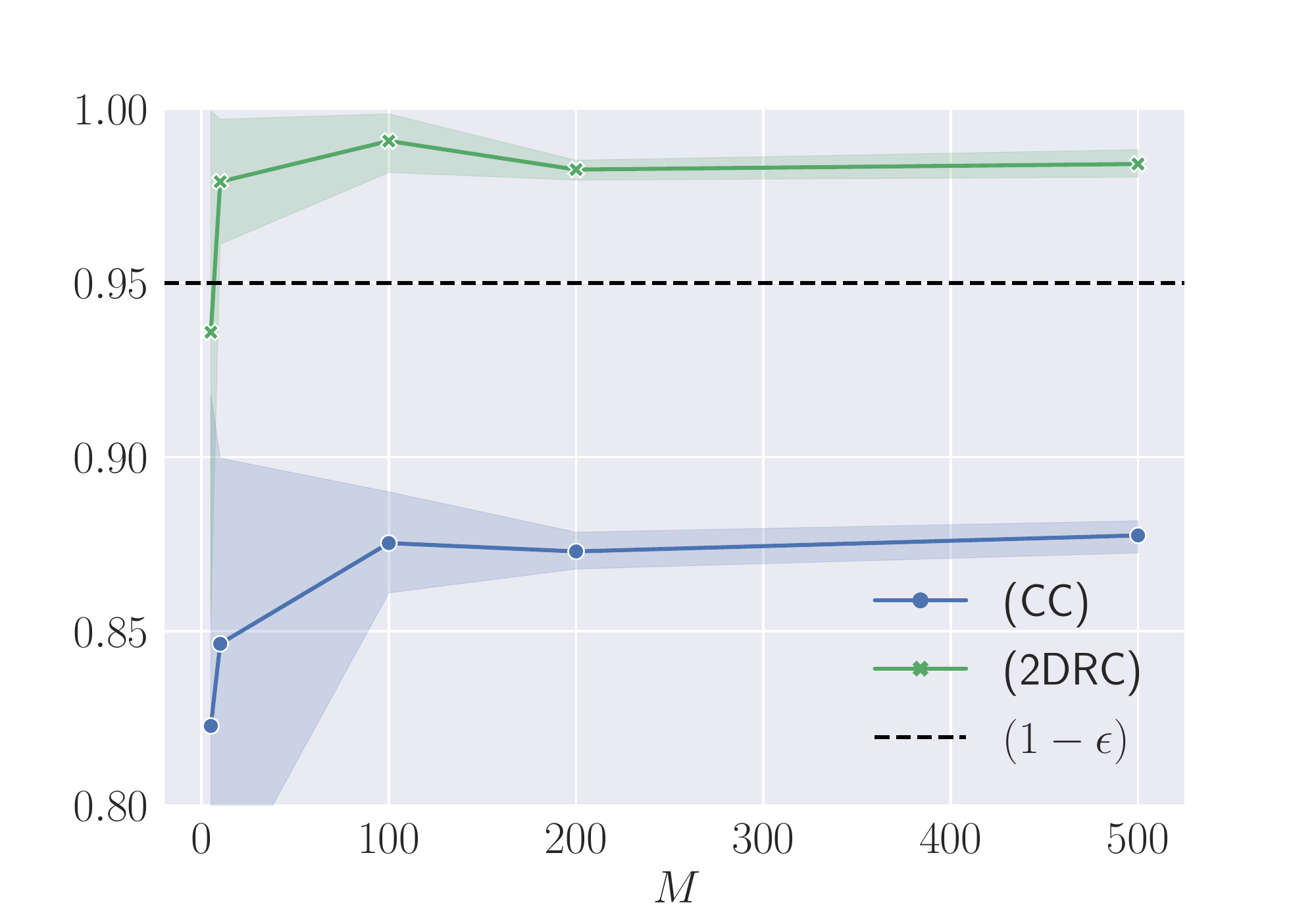}}
\\
\subfloat[\(M = 100\)\label{fig:exps-non-Gauss-oos-vs-delta}]{%
\includegraphics[width=0.8\linewidth]{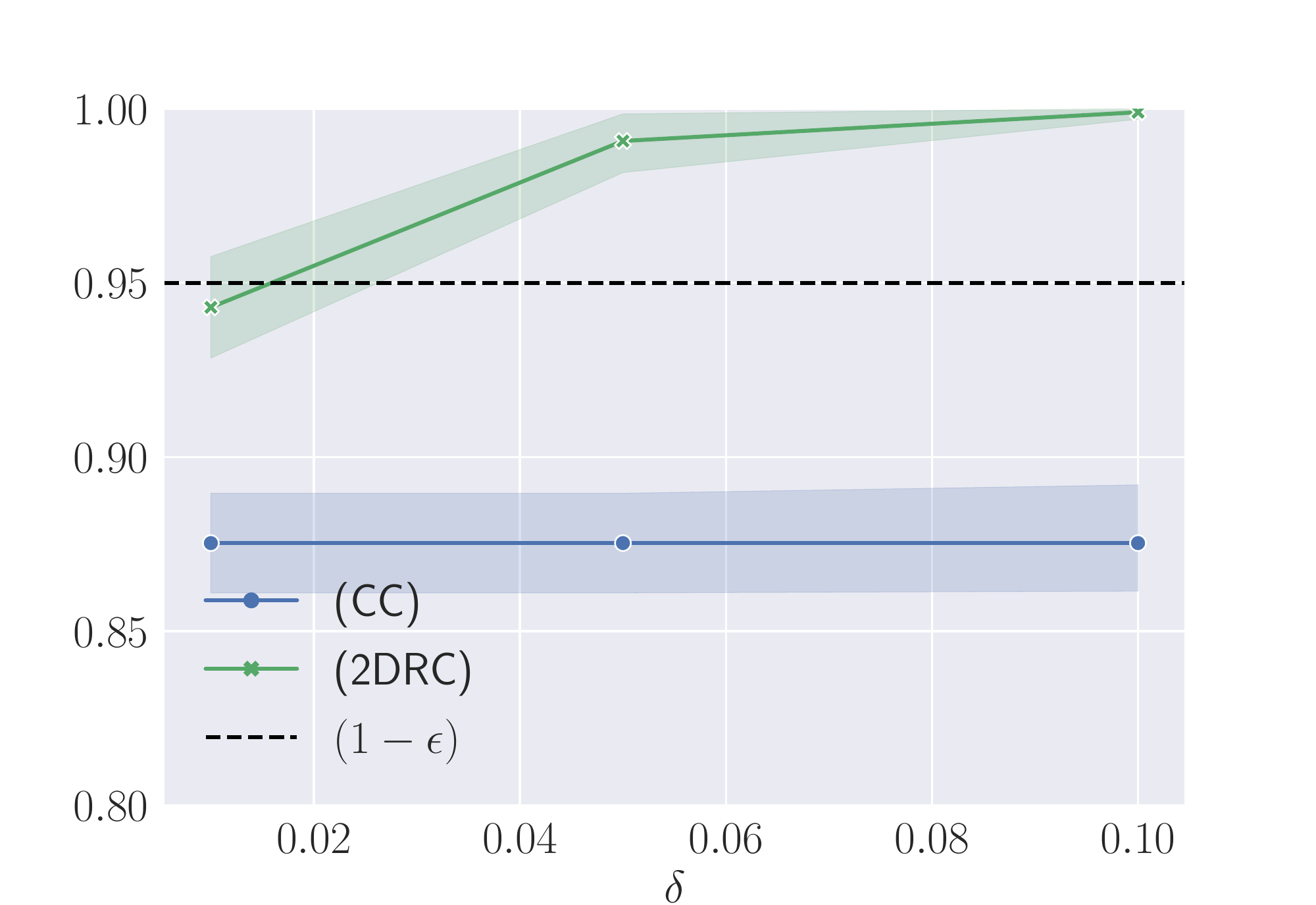}
}
\caption{OOS performance on different problem sizes.}
\end{figure}

Next, we demonstrate the strength of the proposed inner approximation \(\hat{\mathcal{Z}}^N_0\) on a problem instance with
\(\epsilon = 0.05, \delta = 0.08, M = 100\). From Fig.~\ref{fig:exps-opt-oos-vs-N}, we observe that as the number of pieces \(N\) increases the OOS decreases but still remains above the target threshold of \(95\%\), while the optimal value (OPT) of~\eqref{eq:exps-ts-drc} improves. This makes sense because \(\hat{\mathcal{Z}}^N_0\) becomes tighter as \(N\) increases, as promised by Theorem~\ref{thm:non-asymptotic-bd}. Nonetheless, Fig.~\ref{fig:exps-opt-oos-vs-N} also indicates that the change in OOS and OPT is quite limited as \(N\) increases, implying that in reality a small \(N\) can already lead to an excellent approximation.
\begin{figure}[!htbp]
\centering
\includegraphics[width=0.8\linewidth]{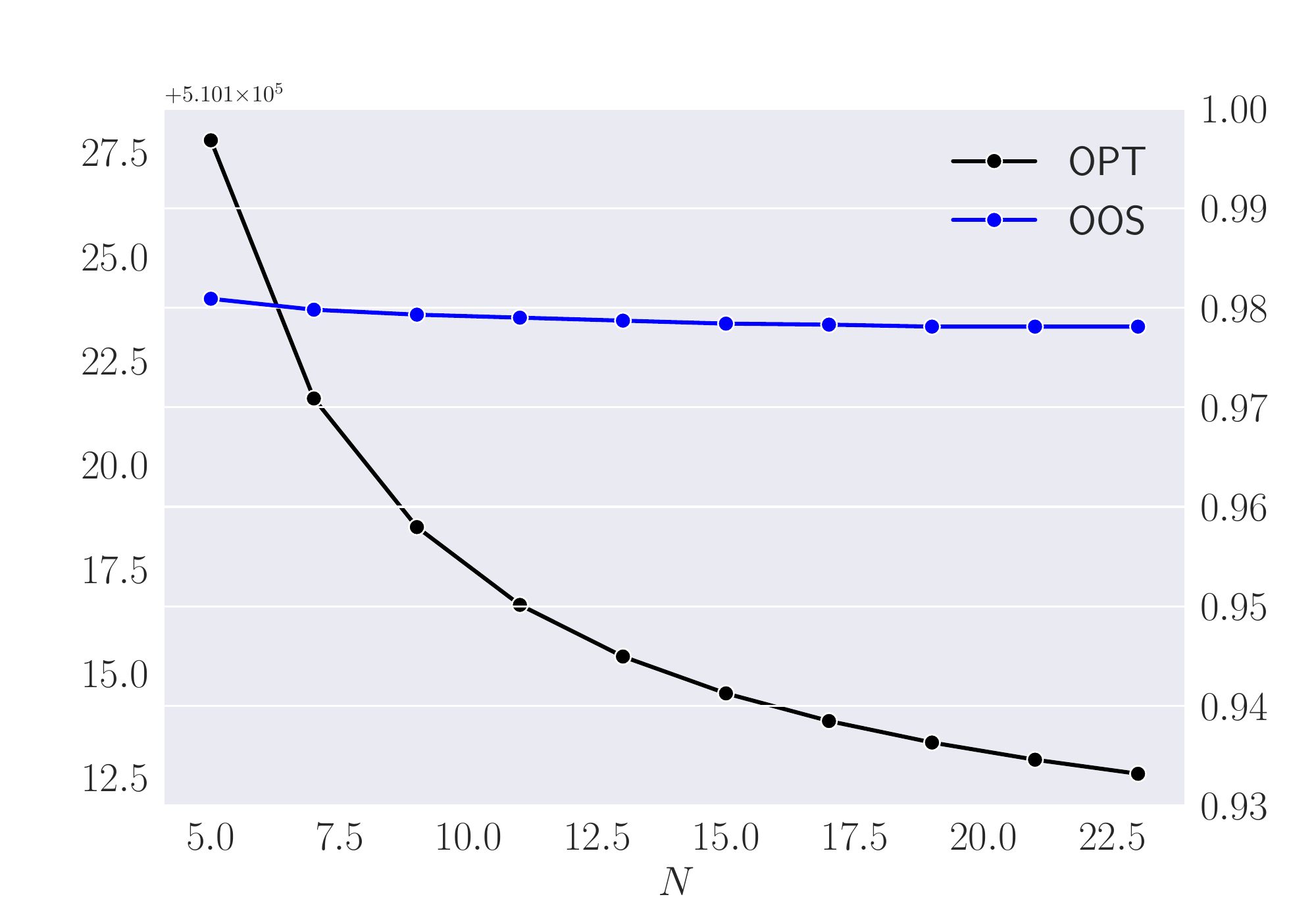}
\caption{Change in OPT and OOS when \(N\) increases.}\label{fig:exps-opt-oos-vs-N}
\end{figure}

Finally, we demonstrate the scalability of our approach using IEEE systems with various sizes. We report the sizes and run time of these instances in Table~\ref{tab:exps-test-time}. These results show that the run time increases mildly as the size of the instance increases. For example, we are able to solve IEEE instances with \(2,000+\) and \(3,000+\) buses within \(10\) seconds.
\begin{table}[!htbp]
\centering
\caption{Computational Time on IEEE systems with various sizes with \(\epsilon = 0.1, \delta = 0.5, N = 7, M = 1000\)}
\label{tab:exps-test-time}
\begin{tabularx}{1.0\linewidth}{cccccc}
  \toprule
                        & case30 & case39 & case118 & case2383 & case3120 \\
  \midrule
  \(\abs{\mathcal{B}}\) & 30     & 39     & 118     & 2383     & 3120     \\
  \(\abs{\mathcal{E}}\) & 41     & 46     & 186     & 2896     & 3693     \\
  \(\abs{\mathcal{G}}\) & 6      & 10     & 54      & 327      & 505      \\
  \# of renewables      & 3      & 4      & 4       & 10       & 10       \\
  \midrule
  Time (sec)            & 0.028  & 0.099  & 0.172   & 6.852    & 8.201    \\
  \bottomrule
\end{tabularx}
\end{table}

\printbibliography{}

\end{document}